\title{On linear series with negative Brill-Noether number}
\author{Nathan Pflueger}
\documentclass{article}
\usepackage{fullpage}
\usepackage{amsmath}
\usepackage{amsthm}
\usepackage{amssymb}
\usepackage{array}
\usepackage{delarray}
\usepackage[dvips]{graphics}
\usepackage{epsfig}
\usepackage{color}
\usepackage{fancyvrb}
\usepackage{tikz}
\usetikzlibrary{decorations.pathreplacing}

\usepackage{setspace}

\newcommand{\cc}{\mathcal{C}}

\newcommand{\co}{\mathcal{O}}
\newcommand{\cl}{\mathcal{L}}

\newcommand{\cg}{\mathcal{G}}
\newcommand{\fg}{\mathfrak{g}}
\newcommand{\cm}{\mathcal{M}}
\newcommand{\cw}{\mathcal{W}}

\newtheorem{thm}{Theorem}[section]
\newtheorem{lemma}[thm]{Lemma}
\newtheorem{prop}[thm]{Proposition}
\newtheorem{cor}[thm]{Corollary}

\theoremstyle{definition}
\newtheorem{defn}[thm]{Definition}

\theoremstyle{remark}
\newtheorem{remark}[thm]{Remark}

\newtheorem{eg}[thm]{Example}
\newtheorem{conj}[thm]{Conjecture}
\newtheorem{qu}[thm]{Question}

\begin{document}
\maketitle

\begin{abstract}
Brill-Noether theory studies the existence and deformations of curves in projective spaces; its basic object of study is $\cw^r_{d,g}$, the moduli space of smooth genus $g$ curves with a choice of degree $d$ line bundle having at least $(r+1)$ independent global sections. The Brill-Noether theorem asserts that the map $\cw^r_{d,g} \rightarrow \cm_g$ is surjective with general fiber dimension given by the number $\rho = g - (r+1)(g-d+r)$, under the hypothesis that $0 \leq \rho \leq g$. One may naturally conjecture that for $\rho < 0$, this map is generically finite onto a subvariety of codimension $-\rho$ in $\cm_g$. This conjecture fails in general, but seemingly only when $-\rho$ is large compared to $g$. This paper proves that this conjecture does hold for at least one irreducible component of $\cw^r_{d,g}$, under the hypothesis that $0 < -\rho \leq \frac{r}{r+2} g - 3r+3$. We conjecture that this result should hold for all $0 < -\rho \leq g + C$ for some constant $C$, and we give a purely combinatorial conjecture that would imply this stronger result.
\end{abstract}



\section{Introduction}\label{intro}

Throughout this paper, a \textit{curve} will always mean a complete algebraic curve over $\textbf{C}$, with at worst nodes as singularities.

Brill-Noether theory studies the ways curves can lie in projective spaces. One of the principle objects of study is the moduli space $\cw^r_{d,g}$, which parameterizes curves of genus $g$ together with a chosen line degree $d$ line bundle with at least $(r+1)$ independent global sections. The geometry of this space is well-understood over general curves; in particular the Brill-Noether theorem \cite{GH} states that when $r \geq 0$ and $g-d+r \geq 1$ a general fiber of the map $\cw^r_{d,g} \rightarrow \cm_g$ is either empty or has dimension given by the \textit{Brill-Noether number}, traditionally denoted $\rho$ and defined as follows.

\begin{equation*}
\rho(g,d,r) = g - (r+1)(g-d+r)
\end{equation*}

Furthermore, the general fiber is empty if and only if $\rho \geq 0$. This paper considers the extension of the Brill-Noether theorem to the case $\rho < 0$, i.e. to non-general curves. The main result is the following.

\begin{thm}\label{mainTheorem}
Suppose that $g,d,r$ are positive integers with $g-d+r \geq 2$ and $0 > \rho \geq -\frac{r}{r+2} g + 3r-3$. Then $\cw^r_{d,g}$ has a component of dimension $\dim \cm_g + \rho$, whose image in $\cm_g$ has codimension equal to $-\rho$, and whose general member has rank exactly $r$.
\end{thm}

The assumptions $r \geq 1$ and $g-d+r \geq 2$ are necessary: if $g-d+r \leq 0$ then $\rho > 0$, while if $g-d+r = 1$ or $r=0$ (these situations are dual to each other) then $\cw^r_{d,g}$ is empty if $\rho < 0$.

Note that $\dim \cm_g + \rho$ is a lower bound on the dimension of any component of $\cw^r_{d,g}$ (we will see one proof in section \ref{twisted}, by combining formula \ref{weightBound} and lemma \ref{twpToBN}), so this theorem asserts that this bound is achieved for not-too-negative values of $\rho$. 

The proof proceeds by induction on the genus. The statement of \ref{mainTheorem} is not suitable for induction; we instead introduce the notion of \textit{twisted Weierstrass points}, and prove a suitable generalization in this context. The method is based on the limit linear series techniques introduced by Eisenbud and Harris \cite{EH87} to construct certain Weierstrass points. As a second application of our techniques, we also prove that the naive dimension estimate for the number of moduli of a Weiestrass point always fails when the Semigroup does not satisfy a combinatorial condition called primitivity (Theorem \ref{Imprim}).


The outline of this paper is as follows. Section \ref{context} discusses background and previous results, and states some conjectures. Section \ref{twisted} introduced the notion of a \textit{twisted Weierstrass point} corresponding to a partition $P$; we define moduli spaces $\cw_g(P)$ of twisted Weierstrass points on genus $g$ curves and show that studying $\fg^r_d$s on genus $g$ curves is equivalent to studying twisted Weierstrass points on genus $g$ curves corresponding to the ``box-shaped'' partition $((g-d+r)^{r+1})$. Section \ref{bridges} describes a construction, using the theory of limit linear series, of twisted Weierstrass points in genus $g+1$ from twisted Weierstrass points in genus $g$. Section \ref{difficulty} defines a combinatorial invariant called the \textit{difficulty} of a partition, and shows how bounding this invariant implies the existence of dimensionally proper twisted Weierstrass points. Sections \ref{primitive} and \ref{linearSeries} demonstrate this technique by bounding the difficulty of two different sorts of partitions. Section \ref{primitive} reproves a theorem of Eisenbud and Harris on dimensionally proper Weierstrass points\footnote{The proof is essentially the same as theirs, although we circumvent the analysis of limit canonical series.}, and then proves theorem \ref{Imprim}, showing that a primitivity hypothesis in that theorem cannot be removed. Finally, section \ref{linearSeries} gives a bound on the difficulty of box-shaped partitions sufficient to prove theorem \ref{mainTheorem}.

\section{Background and conjectures}\label{context}





As the numbers $g,d,r$ vary (constrained by $g-d+r \geq 2$), the spaces $\cw^r_{d,g}$ exhibit two very different sorts of behavior. For $0 \leq \rho \leq g$, the situation is well-understood: $\cw^r_{d,g}$ is irreducible, maps subjectively to $\cm_g$, and has general fiber of dimension $\rho$. On the other hand, when $-\rho \gg 0$ the dimension estimate $(3g-3) + \rho$ fails dramatically. Indeed, many natural families of curves (such as complete intersections, determinantal curves, and curves on rational surfaces) have degree and genus such that $\rho$ is extremely negative, and yet these families have rather large dimension. This phenomenon, observed in numerous examples, has led to the following folklore conjecture, sometimes called the \textit{rigid curves conjecture.}\footnote{This conjecture is usually phrased in terms of components of the Hilbert scheme, but this form is essentially the same.}

\begin{conj}
For all $r$, there is a positive number $C(r)$ such that whenever $\cw^r_{d,g}$ is nonempty, all of its components have dimension at least $C(r)g$.
\end{conj}

Observe that since $\dim \cm_g$ is of course $3g-3$, and a genus $g$ curve has a a $g$-dimensional space of line bundles of degree $d$, the dimension of $\cw^r_{d,g}$ is always less than $4g$. So another way to state this conjecture is the following: for there is a positive number $C(r)$ such that whenever $\cw^r_{d,g}$ is nonempty,

\begin{equation*}
C(r) < \frac{\dim \cw^r_{d,g}}{g} < 4
\end{equation*}

(and the same is true for each irreducible component of $\cw^r_{d,g}$).

This conjecture predicts that there is a sort of ``phase transition'' as $\rho$ moved from slightly negative values to very negative values, where the Brill-Noether dimension estimate begins to fail and the natural tendency of embedded curves to vary in families of dimension linear in $g$ (in addition to the $\dim PGL_{r+1}$ degrees of freedom from projective space itself) begins to dominate. The question this paper aims to address is: \textit{where does this phase transition occur}?

Anecdotal evidence suggests that the transition occurs at a constant multiple of $g$. For example, the simplest case of an embedded curve violating the Brill-Noether dimension estimate is the complete intersection of a quadric and a quartic surfaces in $\textbf{P}^3$. In this case, $(g,d,r) = (9, 8, 3)$ so $\rho = -7$ and the expected dimension of $\cw^3_{8,9}$ is $17$, but an elementary calculation shows that in fact $\dim \cw^r_{8,9} = 18$. So this counterexample occurs at $\rho = -g+2$.

Eisenbud and Harris \cite{EH89} proved that when $\rho = -1$, the space $\cw^r_{d,g}$ is irreducible of the expected dimension, and that its image in $\cm_g$ is a divisor. Edidin \cite{E} showed that in the case $\cw^r_{d,g}$ has all components of the expected dimension, mapping finitely to $\cm_g$. Eisenbud and Harris claimed in their initial paper on limit linear series \cite{EH86} a result of the same form as our theorem \ref{mainTheorem} was forthcoming, but never published a proof.

Our theorem \ref{mainTheorem} gives further evidence that the phase transition occurs in the vicinity of $\rho = -g$. Its main defect is that it only asserts the existence of some component of $\cw^r_{d,g}$ that behaves as expected. The reason for this restriction is that our method of proof proceeds by smoothing certain reducible curves; this method cannot detect and components of $\cw^r_{d,g}$ whose images in $\cm_g$ are compact.

Our theorem \ref{Imprim} is the first step towards answering a different but very analogous question about Weierstrass points. Since non-primitive semigroups occur in every genus with weights as low as roughly $\frac12 g$, this shows that the analogous phase transition for Weierstrass points seems to occur when the expected codimension is roughly $\frac12 g$. We elaborate considerably on the analogous questions for Weierstrass points in \cite{EW}.

We conclude this section with a general conjecture uniting questions about $\cw^r_{d,g}$ with questions about Weierstrass points. See the following section for the definition of $\cw_g(P)$ and  an explanation of how it is related to $\cw^r_{d,g}$.

\begin{conj}
Let $P$ be a partition and $g$ a positive integer. Let $X$ be any component of $\cw_g(P)$, regarded as a subvariety of $\mathcal{P}\textrm{ic}^0_g \times_{\cm_g} \cm_{g,1}$. There exist two positive functions $A(r)$ and $B(r)$ of $r$ with $0 < A(r) < B(r) < 4$, such that:
\begin{itemize}
\item $\textrm{codim} X \leq \min(B(r)g, |P|)$.
\item If $|P| \leq A(r) g$, then $\textrm{codim} X = |P|$.
\end{itemize}
\end{conj}

\begin{qu}
In the range $A(r)g \leq |P| \leq B(r)g$, is there a purely combinatorial procedure to determine if $\cw_g(P)$ has any components of codimension $|P|$?
\end{qu}



\section{Twisted Weierstrass points}\label{twisted}

Theorem \ref{mainTheorem} can be deduced from a slightly stronger result about pointed curves. This section defines the relevant notion, that of \textit{twisted Weierstrass points}, and discusses some basic aspects of their moduli and the connection to theorem \ref{mainTheorem}.

Every point $p$ on a smooth curve $C$ determines a numerical semigroup called the \textit{Weierstrass semigroup} of the point; it consists of those integers $n$ such that $C$ has a rational function of degree $n$ whose only pole is at $p$. For all but finitely many points on a given curve $C$, this semigroup is $\{0, g+1, g+2, \cdots\}$; the other points are called \textit{Weierstrass points}. See \cite{dC} for history and applications of Weierstrass points. Consider the following generalization.

Let $C$ be a smooth curve, $\cl$ a degree $0$ line bundle on $C$, and $p \in C$ a point. The \emph{twisted Weierstrass sequence} of the triple $(C,\cl,p)$ is the following set of nonnegative integers.
\begin{equation*}
S(C,\cl,p) = \{ n \in \textbf{Z}_{\geq 0}:\ h^0(\cl(np)) > h^0(\cl((n-1)p)) \}
\end{equation*}

In other words, the twisted Weierstrass sequence is the set of possible pole orders at $p$ of rational sections of $\cl$ that are regular away from $p$. In the special case $\cl = \co_C$, the twisted Weierstrass sequence is the classical Weierstrass semigroup. By the Riemann-Roch formula, the complement of $S$ has precisely $g$ elements, where $g$ is the genus of $C$. If twisted Weierstrass sequences are given the obvious partial ordering, then they are upper semi-continuous families; therefore the \textit{general} twisted Weierstrass sequence is simply

\begin{equation*}
S = \{g, g+1, g+2, \cdots \}.
\end{equation*}

A triple $(C,\cl,p)$ with a different sequence is called a \textit{twisted Weierstrass point.}

We can and will describe a twisted Weierstrass sequence using the (equivalent) data of a partition. Namely, the \emph{twisted Weierstrass partition} $P(C, \cl, p)$ is given by the multiset $\{ (n+g) - s_n  \}$ (restricted to positive entries), where the twisted Weierstrass sequence is $s_0 < s_1 < s_2 < \cdots$. Alternatively, one can identify twisted Weierstrass sequences with Schubert cycles, which are identified with partitions in the usual way. This connection will be made more explicit after the definition below.

\begin{defn}
Given a nonnegative integer $g$ and a partition $P$, let $\tilde{\cw}_g(P)$ denote the moduli space of triples $(C, \cl, p)$, where $C$ is a smooth curve, $\cl$ is a line bundle of degree $0$ and $p \in C$, such that $P(C, \cl, p) = P$. Let $\cw_g(P)$ denote the closure of $\tilde{\cw}_g(P)$ in $\mathcal{P}ic_g^0 \times_{\cm_g} \cm_{g,1}$.
\end{defn}

The space $\tilde{\cw}_g(P)$ can also be described in terms of Schubert cycles, as follows. Any family of curves with marked point and line bundle corresponds to the following data:

\begin{itemize}
\item A family of curves $\pi:\ \cc \rightarrow B$,
\item A section $s:\ B \rightarrow \cc$, and
\item A line bundle $\cl$ on $\cc$.
\end{itemize}

These data determine a filtration of vector bundles on $B$, given by $E_k = \pi_* (\cl((k-1)\Sigma) / \cl(-\Sigma))$, where $\Sigma$ is the divisor given by the image of $s$. By Grauert's theorem (\cite{H} corollary 12.9), each $E_k$ is a vector bundle of rank $k$ on $B$. In addition to these, there is also a rank $g$ vector bundle on $B$ given by $F = \pi_* ( \cl((2g-1)p))$, with an obvious inclusion $F \hookrightarrow E_{2g}$. This inclusion induces a section $t: B \rightarrow G$ to the Grassmannian bundle $G$ of $g$-planes in $E_{2g}$. The filtration of $E_{2g}$ given by $E_0 \subset E_1 \subset \cdots \subset E_{2g}$ defines, for each partition $P$, an open Schubert cycle $\tilde{\Sigma}_P \subseteq G$, of codimension $|P|$ (see \cite{PAG} section $1.5$ for a definition and basic properties of Schubert cycles). Then the points of $B$ corresponding to points of $\tilde{\cw}_g(P)$ are precisely the inverse image $t^{-1}(\Sigma_P)$.

The description of $\tilde{\cw}_g(P)$ in terms of Schubert cycles gives the following bound on its local dimension at any point.

\begin{equation}\label{weightBound}
\dim_{(C, \cl, p)} \tilde{\cw}(P) \geq (4g-2) - |P|
\end{equation}

\begin{defn}
A point $(C, \cl, p) \in \tilde{\cw}_g(P)$ where equality holds in \ref{weightBound} is called a \emph{dimensionally proper point}.
\end{defn}

\begin{eg}
Let $P = (g)$. Then $(C, \cl, p) \in \tilde{\cw}_g(P)$ if and only if $h^0(\cl) = 1$ and $h^0(\cl(gp) = 1)$. This is is true if and only if $\cl = \co_C$ and $p$ is not a Weierstrass point. So $\tilde{\cw}_g(P)$ is isomorphic to the complement in $\cm_{g,1}$ of the locus of Weierstrass points, and $\cw_g(P) \cong \cm_{g,1}$. Therefore the local dimension at each point is $(3g-2) = (4g-2) - |P|$, so every point is dimensionally proper.
\end{eg}

\begin{eg}
Let $P =  (g\ 1)$. Then $\tilde{\cw}_g(P)$ consists of triples $(C, \cl, p)$ such that $h^0(\cl) = 1, h^0((g-1) p) = 1$, and $h^0(gp) = 2$. In other words, this is the locus in $\cm_{g,1}$ of simple Weierstrass points. This is \'{e}tale-locally isomorphic to $\cm_g$, so every point has local dimension $(3g-3) = (4g-2) - |P|$, so all points are dimensionally proper.
\end{eg}

\begin{figure}
\begin{center}
\begin{tabular}{ccc}

\begin{tikzpicture}[scale=0.5]
\draw (0,0) rectangle (7,1);
\draw (1,0) -- (1,1);
\draw (2,0) -- (2,1);
\draw (3,0) -- (3,1);
\draw (4,0) -- (4,1);
\draw (5,0) -- (5,1);
\draw (6,0) -- (6,1);
\draw [decorate,decoration={brace,amplitude=10pt},yshift=-5pt] (7,0) -- (0,0) node [midway, yshift=-15pt] {$g$};
\end{tikzpicture}
&
\begin{tikzpicture}[scale=0.5]
\draw (0,0) rectangle (7,1);
\draw (1,0) -- (1,1);
\draw (2,0) -- (2,1);
\draw (3,0) -- (3,1);
\draw (4,0) -- (4,1);
\draw (5,0) -- (5,1);
\draw (6,0) -- (6,1);
\draw (0,1) -- (0,2) -- (1,2) -- (1,1);
\draw [decorate,decoration={brace,amplitude=10pt},yshift=-5pt] (7,0) -- (0,0) node [midway, yshift=-15pt] {$g$};
\end{tikzpicture}
&
\begin{tikzpicture}[scale=0.5]
\draw (0,0) rectangle (7,4);
\draw (1,0) -- (1,4);
\draw (2,0) -- (2,4);
\draw (3,0) -- (3,4);
\draw (4,0) -- (4,4);
\draw (5,0) -- (5,4);
\draw (6,0) -- (6,4);
\draw (0,1) -- (7,1);
\draw (0,2) -- (7,2);
\draw (0,3) -- (7,3);
\draw [decorate,decoration={brace,amplitude=10pt},yshift=-5pt] (7,0) -- (0,0) node [midway, yshift = -15pt] {$(g-d+r)$};
\draw [decorate,decoration={brace,amplitude=10pt},xshift=-5pt] (0,0) -- (0,4) node [midway, xshift = -30pt] {$(r+1)$};
\end{tikzpicture}

 \\ $\cw_g(P) \cong \cm_{g,1}$ & $\cw_g(P) \cong \{\mbox{Weierstrass points\}}$ & $\cw_g(P) \cong \cw^r_{d,g} \times_{\cm_g} \cm_{g,1}$
\end{tabular}
\caption{Three examples of partitions and the geometric interpretation of $\cw_g(P)$.}
\label{figure:twpExamples}
\end{center}
\end{figure}
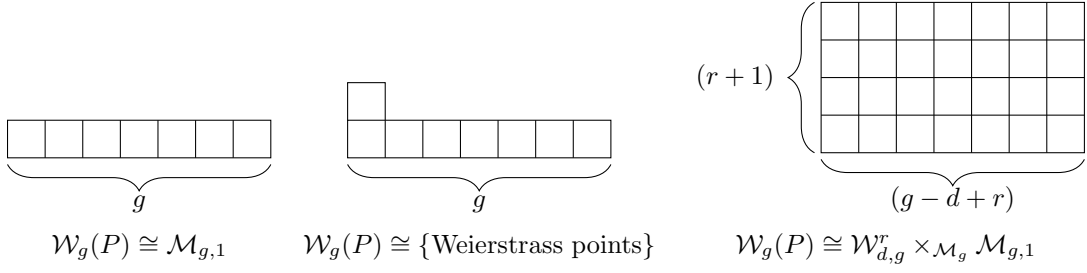

We will now study twisted Weierstrass points with the particular type of partition that will be relevant to theorem \ref{mainTheorem}. Let $P = (m^n)$ (i.e. the number $m$ occurs $n$ times). This partition corresponds to the following twisted Weierstrass sequence.

\begin{equation*}
S = \{g-m, g-m+1, \cdots, g-m+n-2, g-m+n-1, g+n, g+n+1, g+n+2, \cdots\}
\end{equation*}

Then $(C, \cl, p) \in \tilde{\cw}_g(P)$ if and only if the following conditions hold.

\begin{itemize}
\item $h^0(\cl((g-m-1)p)) < h^0(\cl(g-m)p) = 1$
\item $h^0(\cl(g-m+n-1)p) = h^0(\cl(g+n-1)p) = n$
\end{itemize}

These conditions are equivalent to saying that $\cl' = \cl((g+n-1)p)$ is a line bundle of degree $(g-m+n-1)$ and rank $n-1$, such that $p$ is not a ramification point for either the complete linear series $|\cl'|$ or its dual $|\omega_C \otimes \cl'^{\wedge}|$. See \cite{ACGH} appendix C for a definition of ramification points, and a proof that there are finitely many of them for a given linear series. Note that this is not true in positive characteristic.

Since any linear series has a finite number of ramification points, this means that for any line bundle $\cm$ on $C$ of degree $d = (g-m+n-1)$ and rank $r = (n-1)$, the triple $(C, \cm(-dp), p)$ is a point of $\tilde{\cw}_g(P)$ for all but finitely many points $p \in C$. The upshot of this is the following.

\begin{lemma}\label{twpToBN}
Let $g,d,r$ be integers, and let $\tilde{\cw}^r_{d,g} \subset \mathcal{P}ic^d_g$ consist of those pairs $(C, \cl)$ where $\cl$ is a line bundle on $C$ with degree $d$ and $h^0(\cl) = r+1$\footnote{The only difference from the definition of $\cw^r_{d,g}$ is that here exact equality is required.}. Let $P$ be the partition $((g-d+r)^{r+1})$. Then there is a map
\begin{eqnarray*}
f: \tilde{\cw}_g(P) &\rightarrow& \tilde{\cw}^r_{d,g}\\
(C, \cl, p) &\mapsto& (C, \cl(dp))
\end{eqnarray*}
which is surjective, and whose fiber over any point $(C, \cl) \in \tilde{\cw}^r_{d,g}$ is isomorphic to $C$ with finitely many punctures.
\end{lemma}

Notice that, in the notion of the lemma, $|P| = (r+1)(g-d+r) = g-\rho(g,d,r)$. It follows from this that the map $f$ in the lemma sends dimensionally proper points to dimensionally proper points. Thus to study dimensionally proper line bundles on curves is equivalent to studying dimensionally proper twisted Weierstrass points given by ``box-shaped'' partitions. This will be the object of the remainder of the paper.

\begin{remark}\label{twpDuality}
Notice that twisted Weierstrass points have a duality property: namely if $P^*$ is the dual partition of $P$ (that is, $P^*_n = \left| \{m:\ P_m > n \right|$), then $\tilde{\cw}_g(P) \cong \tilde{\cw}(P^*)$, via the map $(C, \cl, p) \mapsto (C, \omega_C(-(2g-2)p) \otimes \cl^\wedge, p)$. This generalizes the fact that $\tilde{\cw}^r_{d,g} \cong \tilde{\cw}^{g-d+r-1}_{(2g-2)-d, g}$ via $\cl \mapsto \omega_C \otimes \cl^\wedge$ (via the correspondence discussed above), since the dual partition of $((g-d+r)^{r+1})$ is $((r+1)^{g-d+r})$. This duality is reflected, for example, in the two perspectives by which one typically studies classical Weierstrass points: in terms of pole order of rational functions or in terms of ramification of the canonical series.
\end{remark}

\begin{qu}
Let $\mu(P,g)$ be the maximum codimension of a component of $\cw_g(P)$ (or $-\infty$ if there are none).  When is $\mu(P,g) < |P|$? Is there a purely combinatorial description of which partitions $P$ and integers $g$ give strict inequality?
\end{qu}

We will define in section \ref{difficulty} a function $\delta(P)$ of partitions such that $\mu(P,g) = |P|$ whenever $g \geq \frac12 (|P| + \delta(P))$. Bounding this function will give theorem \ref{mainTheorem}. First we describe the smoothing argument which underlies the definition of $\delta(P)$.

\section{Limits of twisted Weierstrass points}\label{limits}

To prove the existence of certain twisted Weierstrass points, it will be necessary to allow the curves to degenerate to singular curves, and to have a suitable notion of limits of the twisted Weierstrass points. Such a notion is provided by limit linear series, as introduced by Eisenbud and Harris \cite{EH86}. We begin by recalling the relevant definitions; see \cite{HM} for an expository treatment.

A \emph{linear series} of degree $d$ and rank $r$ on a smooth curve $C$, also called a $\mathfrak{g}^r_d$, is a pair $L = (\cl, V)$, where $\cl$ is a degree $d$ line bundle and $V \subseteq H^0(\cl)$ is an $(r+1)$-dimensional vector space of sections. The moduli space of genus $g$ curves with a chosen $\fg^r_d$ is denoted $\cg^r_{d,g}$. Given a linear series $L$ and a point $p \in C$, the \emph{vanishing sequence} of $L$ at $p$ is the set of integers $n$ such that $V$ contains a section vanishing to order exactly $n$ at $p$. This sequence consists of $(r+1)$ distinct integers; it is usually denoted $a^L(p) = (a^L_0(p), a_1^L(p), \cdots, a_r^L(p))$ where $a_0^L(p) < a_1^L(p) < \cdots < a_r^L(p)$. Equivalent to the vanishing sequence is the \emph{ramification sequence} $\alpha^L(p)$, given by $\alpha^L_i(p) = a^L_i(p) - i$. Most authors work with the ramification sequence rather than the vanishing sequence; we will work almost entirely with the vanishing sequence since it is slightly more notationally convenient for our purposes.

Let $\tilde{\cg}^r_{d,g} (a) \subseteq \cg^r_{d,g} \times_{\cm_g} \cm_{g,1}$ denote the space of triples $(C,L,p)$ such that the vanishing sequence of $L$ at $p$ is precisely $a$. Let $\cg^r_{d,g}(a)$ denote the space of such triples such that the vanishing sequence of $L$ at $p$ is at least\footnote{Whenever we say that a sequence $a$  is ``at least'' another sequence $a'$, we mean that $a_i \geq a_i'$ for each $i$.} $a$. More generally, $\tilde{\cg}^r_{d,g} (a^1, a^2, \cdots, a^s) \subseteq \cg^r_{d,g} \times_{\cm_g} \cm_{g,s}$ denotes the space of tuples $(C,L,p_1, \cdots, p_s)$ with vanishing sequence $a^i$ at $p_i$.

The theory of limit linear series works best for curves of compact type. A nodal curve $X$ is called \emph{compact type} if its dual graph (that is, the graph whose vertices are the components of $X$ and whose edges correspond to the nodes) has no cycles (equivalently, the Jacobian of $X$ is compact). Recently, Amini and Baker gave a definition of limit linear series for arbitrary nodal curves, but there does not yet exist a moduli space for these more general limit linear series. We will use the original definitions of Eisenbud and Harris.

\begin{defn}
Let $X$ be a curve of compact type. A \emph{refined limit linear series} $L$ of degree $d$ and rank $r$ (or \emph{limit} $\fg^r_d$) on $X$ consists of a $\fg^r_d$ $L^C$ on each connected component $C$ of $X$ (called the $C$-aspect of $L$), such that for each node $p \in X$ joining components $C_1$ and $C_2$, the following compatibility condition holds.

\begin{equation*}
a^{L^{C_1}}_i(p) + a^{L^{C_2}}_{r-i}(p) = d \mbox{ for } i=0,1,2,\cdots,r
\end{equation*}

The \emph{vanishing sequence} $a^L(p)$ of a limit series at a smooth point $p$ is the vanishing sequence of the $C$-aspect of $L$, where $p \in C$.
\end{defn}

Eisenbud and Harris also define \emph{coarse} limit series to be a collection of $C$-aspects such that the compatibility condition holds as an inequality. We will not need to consider coarse limits in this paper. Note that Osserman \cite{O} gave a different definition of limit linear series that is more suitable for the construction of a global moduli scheme. His definition is equivalent to the Eisenbud-Harris definition in the special case of refined limit series. Eisenbud and Harris do not construct a global moduli space of limit $\fg^r_d$s over all of $\cm_g$, but instead construct a local moduli space. More precisely, they construct a moduli space of (refined) limit linear series over a Kuranishi family of any curve of compact type. In either formalism, the existence of a suitable moduli space, plus a dimension bound on it coming from Schubert conditions, implies the following ``regeneration theorem.'' To state it first requires one more definition.

\begin{defn}
A marked curve $(X, p_1, \cdots, p_s)$ with a linear series $L$ of degree $d$ and rank $r$ is called \emph{dimensionally proper} if the local dimension of $\cg^r_{d,g}(a^L(p_1), \cdots, a^L(p_s))$ is exactly 

\begin{equation*}
\dim \cm_{g,s} + \rho - \sum_{i=1}^s \sum_{j=0}^r (a^L_j(p_i) - j).
\end{equation*}
\end{defn}

\begin{thm}[Corollary 3.7 of \cite{EH86}]  \label{regeneration}
Let $L$ be a limit $\fg^r_d$ on a curve $X$ of compact type, and $p_1, \cdots, p_s \in X$ are smooth points. Suppose that each component $C$ of $X$ is dimensionally proper with respect to all the points of $C$ that are nodes in $X$ and all the marked points $p_i$ that lie on $C$. Then there exists a smooth marked curve $(X', p_1', \cdots, p_s')$ with a dimensionally proper $\fg^r_d$ $L'$ such that $a^{L'}(p_j') = a^L(p_j)$ for all $j$. This marked curve and linear series lies in a one-parameter family whose limit is the marked curve $(X, p_1, \cdots, p_s)$ with limit linear series $L$.
\end{thm}

\begin{defn}\label{dplls}
A marked curve $(X,p_1, \cdots, p_s)$ of compact type with a refined limit $L$ satisfying the hypotheses of theorem \ref{regeneration} will also be called \emph{dimensionally proper}.
\end{defn}

The following lemma reinterprets the data of a twisted Weierstrass point in a manner that makes the theory of limit linear series applicable.

\begin{lemma}\label{twpgdr}
Let $r \geq g-1$ be an integer. Then $\tilde{\cw}_g(P) \cong \tilde{\cg}^r_{r+g,g}(a)$, where $a = (a_0, a_1, \cdots, a_r)$ is the sequence given by $a_i = i + P_{r-i}$, via the maps $(C, \cl, p) \mapsto (C, |\cl((r+g)p)|, p)$ and $(C, (\cl, H^0(\cl)), p) \mapsto (C, \cl(-(r+g)p),p)$.
\end{lemma}
\begin{proof}
Since $r+g \geq 2g-1$, $|\cl((r+g)p)|$ is indeed a $\fg_{r+g}^r$; unraveling definitions shows that the vanishing sequence at $p$ is $a$. So this is a well-defined map to  $\tilde{\cg}^r_{r+g,g}(a)$. In reverse, every $\fg^r_{r+g}$ is necessarily complete, hence of the form $|\cl|$ for some $\cl$; then  $(C, \cl(-(r+g)p),p)$ indeed lies in $\tilde{\cw}_g(P)$ by the same calculation.
\end{proof}

Therefore, we have the following notion of a \emph{limit twisted Weierstrass point}: a curve $X$ of compact type, with marked smooth point $p$ and refined limit $\fg^r_{r+g}$ $L$ (where $r \geq g-1$) and the vanishing sequence $a$ as described above. Constructing such object, and proving that they are dimensionally proper (in the sense of definition \ref{dplls}) will suffice to construct dimensionally proper twisted Weierstrass points (on smooth curves).

\section{Elliptic bridges and displacement}\label{bridges}

The object of this section is to demonstrate how dimensionally proper twisted Weierstrass points on genus $g$ curves give rise to dimensionally proper twisted Weierstrass points on curves of genus $g+1$, with slightly modified partitions. The construction proceeds by adjoining an elliptic curve to the genus $g$ curve, and smoothing the resulting nodal curve. The basic technical tool is the regeneration theorem for limit linear series, as introduced by Eisenbud and Harris \cite{EH86} (see \cite{HM} for a readable expository account and \cite{O} for a more recent perspective that is more applicable in characteristic $p$).

The following lemma is a slight restatement of proposition 5.2 from \cite{EH87}. It is the basic tool in our inductive constructions.

\begin{lemma} \label{bridgeLemma}
Fix integers $r,d$ and two sequences $b = (b_0, b_1, \cdots, b_r)$ and $c = (c_0, c_1, \cdots, c_r)$ such that
		
\begin{equation*}
b_{i} +c_{r-i} = d-1
\end{equation*}
		
for each index $i$. Then for any genus $1$ curve $E$ with distinct points $p,q$ and degree $d$ line bundle $\cl$, there exists a unique linear series $L = (\cl, V)$ on $E$ such that for all $i$ the following inequalities hold.
		
\begin{eqnarray*}
a^L_i(p) &\geq& b_i\\
a^L_i(q) &\geq& c_i\\
\end{eqnarray*}
\end{lemma}
	
\begin{proof}
For all pairs of indices $(i,j)$ with $i+j<d$, define the following vector space of sections of $\cl$.
\begin{equation*}
W_{i,j} = \textrm{im} \left( H^0(\cl(-i p - j q )) \hookrightarrow H^0(\cl) \right)
\end{equation*}
		
That is, $W_{i,j}$ consists of those sections vanishing to order at least $i$ at $p$ and at least order $j$ at $q$. $W_{i,j}$ has dimension $d - i - j$, by Riemann-Roch. 

Separate the indices $\{0, 1, 2, \cdots, r\}$ into the longest intervals intervals $I_k = \{u_k, u_k+1, \cdots, v_k\}$, such that $b_{u_k}, b_{u_k+1}, \cdots, b_{v_k}$ are consecutive integers. Let $m$ be the number of these intervals, so that $\{0, 1, \cdots, r\}$ is a disjoint union of $I_1, I_2, \cdots, I_m$. Then $u_1 = 0, v_m = r$, and $v_k + 1 = u_{k+1}$. Define, for each $k \in \{1, 2, \cdots, m\}$, $V_k := W_{b_{u_k}, c_{r-v_k}}$. Observe that the dimension of $V_k$ is $ d- b_{u_k} - b_{r - v_k} = d - b_{u_k} - (d-1) + b_{v_k} = 1 + b_{v_k} - b_{u_k} = 1 + v_k - u_k = |I_k|$.

Let $V$ be the sum of all the spaces $V_k$. We claim that $V$ satisfies the conditions of the lemma, and that it is the unique such vector space of sections. First, we verify that $V$ satisfies the conditions of the lemma. By the Riemann-Roch formula, each vector space $V_k$ has the following orders of vanishing at $p$: $\{b_{u_k}, b_{u_k+1}, \cdots, b_{v_k-1}, b_{v_k}'\}$, where 

\begin{equation*}
b_{v_k}' = \left\{ \begin{array}{ll} b_{v_k}+1 & \textrm{if }\ \cl \cong \co_E((b_{v_k}+1)p + (c_{r-v_k})q) \\ v_k & \textrm{otherwise.} \end{array} \right.
\end{equation*}

In all cases, $b_{v_k}' < b_{u_{k+1}}$, so the sections of any two different spaces $V_k$ have disjoint sets of orders of vanishing at $p$. It follows that the orders of vanishing at $p$ of sections in $V$ is the disjoint union 

\begin{equation*}
\bigcup_{k=1}^m \{b_{u_k}, b_{u_k+1}, \cdots, b_{v_k-1}, b_{v_k}'\}.
\end{equation*}

In particular, the dimension of $V$ is $\sum_{k=1}^m |I_k| = r+1$, and its vanishing sequence at $p$ is at least $b_{u_1},b_{u_1+1}, \cdots, b_{v_1}, b_{u_2}, \cdots, b_{v_2}, \cdots, b_{v_m}$, which is identical to $b_0, b_1, \cdots, b_r$. Symmetric reasoning shows that the orders of vanishing of $V$ at $q$ are at least $c_0, c_1, \cdots, c_r$.  So $V$ satisfies the conditions of the lemma.

Now suppose that $V'$ satisfies the conditions of the lemma. Then the sections of $V'$ vanishing to order at least $a_{u_k}$ have codimension at most $u_k$, and those vanishing to order at least $b_{r-v_k}$ at $q$ have codimension at most $r-v_k$, hence $V' \cap V_k$ has codimension at most $r + (u_k - v_k)$ and thus dimension at least $1 + v_k - u_k = |I_k|$. Therefore this intersection must be all of $V_k$. Thus $V' \supseteq V$, and $\dim V' = \dim V$, so in fact $V' = V$. So $V$ is the unique such vector space of sections.\end{proof}

In fact, examining the end of the proof of lemma \ref{bridgeLemma}, we have actually proved the following.

\begin{lemma} \label{bridge2}
Let $L = (\cl, V)$ be a linear series as described in lemma \ref{bridgeLemma}. Then the actual orders of vanishing of $L$ are as follows:

\begin{eqnarray*}
a^L_i(p) &=& \left\{ \begin{array}{ll} b_i+1 & \textrm{ if } (b_i+1) \in \Lambda \textrm{ and } b_{i+1} > b_i+1 \\ b_i & \textrm{ otherwise} \end{array} \right.\\
a^L_i(q) &=& \left\{ \begin{array}{ll} c_i+1 & \textrm{ if } (c_i+1) \in (d - \Lambda) \textrm{ and } c_{i+1} > c_i+1 \\ c_i & \textrm{ otherwise} \end{array} \right.\\
\end{eqnarray*}

where $\Lambda$ is the arithmetic progression $\{n:\ \cl \cong \co_E(np + (d-n)q) \}$. \hfill $\Box$
\end{lemma}

For notational convenience, we make the following definition. In the following definition and the remainder of this paper, an \textit{arithmetic progression} will be a proper subset $\Lambda$ of the integers such that the set of differences of elements of $\Lambda$ is closed under addition. In particular, $\Lambda$ may be empty or have only a single element, but it may not be all of $\textbf{Z}$.

\begin{defn}
Let $a = (a_0, a_1, a_2, \cdots, a_r)$ be a strictly increasing sequence of integers, and let $\Lambda$ be an arithmetic progression (as defined above). Define the \emph{upward displacement} $a^+_\Lambda$ and \emph{downward displacement} $a^-_\Lambda$ of $a$ with respect to $\Lambda$ as follows.
\begin{eqnarray*}
(a^+_\Lambda)_i &=& \begin{cases} a_i+1 & \mbox{if } a_i+1 \in \Lambda  \mbox{ and } a_{i+1} > a_i+1\\ a_i & \mbox{otherwise}\end{cases}\\
(a^-_\Lambda)_i &=& \begin{cases} a_i-1 & \mbox{if } a_i \in \Lambda \mbox{ and } a_{i-1} < a_i-1  \\ a_i & \mbox{otherwise} \end{cases}
\end{eqnarray*}
In these expressions $i$ is an index in $\{0, 1, \cdots, r\}$ and for notational convenience $a_{-1} = -\infty$ and $a_{r+1} = \infty$ (when these appear on the right side). This definition is interpreted visually, using partition notation, in figure \ref{figure:corners}.
\end{defn}

Informally, the upward displacement ``attracts'' the sequence upward to the progression $\Lambda$, while the downward displacement ``repels'' the sequence downward away from $\Lambda$. Another interpretation is that displacement forgets, for each pair $\{\lambda-1, \lambda\}$ (where $\lambda \in \Lambda$) which of these two numbers is in the sequence, remembering only how many ($0$, $1$, or $2$) are present.

The following lemma reformulates the previous two lemmas in the language of limit linear series.

\begin{lemma}\label{bridge3}
Let $C$ be a smooth curve, $p_1 \in C$ a point, $E$ a genus $1$ curve, and $p_2, q$ two distinct points on $E$. Let $X$ be the the nodal curve obtained by attaching $C$ and $E$ at $p_1$ and $p_2$. Let $L^C = (\cl^C, V^C)$ be a $\fg_d^r$ on $C$, and $\cl^E$ a degree $(d+1)$ line bundle on $E$. Then there exists a unique limit $\fg^r_{d+1}$ $L$ on $X$ with the following properties.
\begin{enumerate}
\item The $C$-aspect of $L$ is $L^C+p_1$ (that is, $L^C$ with a base point added at $p_1$).
\item The $E$-aspect of $L$ has line bundle $\cl^E$.
\item For all $i \in \{0, 1, \cdots, r\}$, $a^L_i(q) \geq a^{L^C}_i(p_i)$.
\end{enumerate}
Let $\Lambda = \left\{n: \cl^E \cong \co_E(nq + (d+1-n)p_2) \right\}$; then the vanishing sequence of $L$ is precisely $a^L(q) = \left( a^{L_C}(p_1) \right)^+_\Lambda$, and $L$ is a refined limit linear series if and only if $(a^{L^C}(p_1))^-_\Lambda = a^{L^C}(p_1)$.
\end{lemma}

\begin{proof}
For $i \in \{0, 1, \cdots, r\}$, let $b_i =  d - a_{r-i}^{L^C}(p_1)$ and let $c_i = a_i^{L^C}(p_1)$. Then of course $b_i + c_{r-i} = (d+1)-1$ for each $i$, so a suitable $E$-aspect for $L$ exists and is unique by lemma \ref{bridgeLemma}. Lemma \ref{bridge2} shows that vanishing sequence of this $E$-aspect (and therefore of $L$) is $c^+_\Lambda$ as claimed. The vanishing sequence of the $E$-aspect at $p_2$ is $b^+_{(d+1-\Lambda)}$, and $L$ is refined if and only if this is equal to $b$. But observe that since $b_i = d - c_{r-i}$, this is equivalent to $c^-_\Lambda = c$, as claimed.
\end{proof}

By allowing the curve $X$ and limit linear series $L$ to vary, this construction on two-component curves gives the following result on dimensionally proper linear series with specified ramification.

\begin{prop}\label{displace}
Suppose that $a = (a_0, a_1, \cdots, a_r)$ is a strictly increasing sequence of nonnegative integers, and $\Lambda$ is an arithmetic progression (as defined above) such that $a^-_\Lambda = a$ and $a^+_\Lambda$ differs from $a$ in at most two places. If $\tilde{\cg}^r_{d,g}(a)$ has a dimensionally proper point, belonging to a component mapping to $\cm_{g,1}$ with general fiber dimension $d$, then $\tilde{\cg}^r_{d+1,g+1}(a^+_\Lambda)$ has a dimensionally proper point, belonging to a connected component mapping to $\cm_{g,1}$ with general fiber dimension at most $d+1$ (if $a = a^+_\Lambda$) or at most $d$ (if $a \neq a^+_\Lambda$).
\end{prop}

\begin{proof}
Assume without loss of generality that $\Lambda$ is as small as possible. This means that if $a^+_\Lambda$ differs from $a$ in two places, then $\Lambda$ is the progression generated by those two values $(a^+_\Lambda)_i$ that are greater than $a_i$; if  $a^+_\Lambda$ differs from $a$ in one place, then $\Lambda$ is a single element; and if  $a^+_\Lambda = a$, then $\Lambda$ is empty.

Let $(C,L^C,p_1)$ be a dimensionally proper point of $\tilde{\cg}^r_{d,g}(a)$. Let $(E, \cl,p_2, q)$ be a twice-pointed elliptic curve, chosen in the following way.
\begin{itemize}
\item If $\Lambda$ is infinite, say $\{n:\ n \equiv m \mod d\}$ for some $m$ and $d$, then let $\cl = \co_E(mq + (d+1-m)p_2)$ and select $p_2, q$ so that $(p_2 - q)$ is a $d$-torsion point on $\textrm{Pic}^0(E)$.
\item If $\Lambda$ has a single element $m$, then let $\cl = \co_E(mq + (d+1-m)p_2)$ and choose $p_2, q$ so that $(p_2-q)$ is not torsion.
\item If $\Lambda$ is empty, then choose $\cl$  distinct from all line bundles $\co_E(mq + (d+1-m)p_2)$ and choose $p_2, q$ arbitrarily.
\end{itemize}
Let $X$ be the nodal curve described in lemma \ref{bridge3}, $L$ the limit $\fg^r_{d+1}$ on $X$ described in that lemma, and $L^E$ its $E$-aspect. Since $a^-_\Lambda = a$, this series is refined. We shall show that $(X,L,q)$ is dimensionally proper in the sense described in section {limits}. By assumption, the $C$-aspect $L^C+p_1$, with the marked point $p_1$, is dimensionally proper. So it suffices to prove that $(E, L^E, p_2, q)$ is dimensionally proper. Let $\delta$ be the number of places where $a^+_\Lambda$ differs from $a$. An elementary calculation shows that this is equivalent to showing that the local dimension of $\cg^r_{d+1,1}\left(a^{L^E}(p_2), a^{L^E}(q)\right)$ at $(E, L^E, p_2, q)$ is $3 - \delta$. Now, the map $f:\ \cg^r_{d+1,1}\left(a^{L^E}(p_2), a^{L^E}(q)\right) \rightarrow \textrm{Pic}^{d+1}_{1, 2}$ (that is, to the moduli space of twice-marked genus $1$ smooth curves with a chosen degree $(d+1)$ line bundle) is set-theoretically injective by lemma \ref{bridgeLemma}. By lemma \ref{bridge2}, the image of $f$ consists of all $(E', \cl', p_2', q')$ such that the arithmetic progression $\Lambda' = \{n:\ \cl' \cong \co_{E'}(nq' + (d+1-n)p_2'\}$ contains $\Lambda$. By a little casework, the dimension of the image is $3- \delta$. It follows that $(E, L^E, p_2,q)$ is dimensionally proper, and therefore so is $(X,L,q)$. By theorem \ref{regeneration}, $\tilde{\cg}^r_{d+1,g+1}(a^+_\Lambda)$ has a dimensionally proper point. The bound on the dimension of fibers over $\cm_{g,1}$ follows by considering the semicontinuity of fiber dimension for the map from the space of limit linear series (on a $1$-parameter family degenerating to $(X,p)$) over $\bar{\cm}_{g,1}$.
\end{proof}

\section{The displacement difficulty of a partition}\label{difficulty}

As before, we will use the following convention: an \textit{arithmetic progression} will mean a proper subset $\Lambda \subset \textbf{Z}$ such that $\Lambda - \Lambda$ is closed under addition. In particular, $\Lambda$ may be empty or have a single element, but it cannot be all of $\textbf{Z}$. Also, we adopt the following notational conventions: the partition elements are $P_0 \geq P_1 \geq \cdots \geq P_n$, and $P_k$ is defined to be $0$ for $k>n$ and $\infty$ for $k < 0$.

\begin{defn}
Let $P$ be a partition and $\Lambda$ an arithmetic progression. Then define the \emph{upward displacement} $P^+_\Lambda$ and \emph{downward displacement} $P^-_\Lambda$ of $P$ with respect to $\Lambda$ as follows.
\begin{eqnarray*}
(P^+_\Lambda)_i &=& \left\{  \begin{array}{ll} P_i + 1 & \textrm{if } (P_i-i) \in \Lambda \textrm{ and } P_{i-1} > P_i \\ P_i & \textrm{otherwise} \end{array} \right.\\
(P^-_\Lambda)_i &=& \left\{  \begin{array}{ll} P_i - 1 & \textrm{if } (P_i-i-1) \in \Lambda \textrm{ and } P_{i+1} < P_i \\ P_i & \textrm{otherwise} \end{array} \right.\\
\end{eqnarray*}
\end{defn}

This definition is much easier to understand visually; it is illustrated in figure \ref{figure:corners}. Here the partition $P$ is represented by its Young diagram, and the arithmetic progression $\Lambda$ is represented by an evenly spaced sequence of diagonal lines. Then the two displacements are obtained by finding all places where the line of $\Lambda$ meet the corners of $P$, and either ``turning the corners out'' (in the case of $P^+_\Lambda$ or ``turning the corners in'' in the case of $P^-_\Lambda$).

\begin{figure}
\begin{center}
\begin{tikzpicture}[scale=0.5]
\draw (0,5) -- (0,0) -- (8,0);
\draw (0,5) -- (1,5) -- (1,0);
\draw (0,4) -- (1,4);
\draw (0,3) -- (1,3);
\draw (0,2) -- (7,2) -- (7,0);
\draw (2,2) -- (2,0);
\draw (3,2) -- (3,0);
\draw (4,2) -- (4,0);
\draw (5,2) -- (5,0);
\draw (6,2) -- (6,0);
\draw (0,1) -- (8,1) -- (8,0);
\draw[style=dashed] (-1,3) -- (2,6);
\draw[style=dashed] (-1,0) -- (5,6);
\draw[style=dashed] (1,-1) -- (8,6);
\draw[style=dashed] (4,-1) -- (9,4);
\draw[style=dashed] (7,-1) -- (9,1);
\draw[style=ultra thick] (0,5) -- (1,5) -- (1,4);
\draw[style=ultra thick] (1,3) -- (1,2) -- (2,2);
\draw[style=ultra thick] (6,2) -- (7,2) -- (7,1);
\draw[style=ultra thick] (8,1) -- (8,0) -- (9,0);
\draw (4,-1) node[below] {$P = (8,7,1,1,1)$};
\draw[->] (0,-3) -- (-2,-5);
\draw[->] (7,-3) -- (9,-5);

\begin{scope}[xshift=-7cm, yshift=-12cm]
\draw (0,4) -- (0,0) -- (8,0);
\draw (0,4) -- (1,4) -- (1,0);
\draw (0,3) -- (1,3);
\draw (0,2) -- (6,2) -- (6,0);
\draw (2,2) -- (2,0);
\draw (3,2) -- (3,0);
\draw (4,2) -- (4,0);
\draw (5,2) -- (5,0);
\draw (0,1) -- (8,1) -- (8,0);
\draw (7,1) -- (7,0);
\draw[style=dashed] (-1,3) -- (2,6);
\draw[style=dashed] (-1,0) -- (5,6);
\draw[style=dashed] (1,-1) -- (8,6);
\draw[style=dashed] (4,-1) -- (9,4);
\draw[style=dashed] (7,-1) -- (9,1);
\draw[style=ultra thick] (0,5) -- (0,4) -- (1,4);
\draw[style=ultra thick] (1,3) -- (1,2) -- (2,2);
\draw[style=ultra thick] (6,2) -- (6,1) -- (7,1);
\draw[style=ultra thick] (8,1) -- (8,0) -- (9,0);
\draw (4,-1) node[below] {$P^-_\Lambda = (8,6,1,1)$};
\end{scope}

\begin{scope}[xshift=7cm, yshift=-12cm]
\draw (0,5) -- (0,0) -- (9,0);
\draw (0,5) -- (1,5) -- (1,0);
\draw (0,4) -- (1,4);
\draw (0,3) -- (2,3) -- (2,0);
\draw (0,2) -- (7,2) -- (7,0);
\draw (3,2) -- (3,0);
\draw (4,2) -- (4,0);
\draw (5,2) -- (5,0);
\draw (6,2) -- (6,0);
\draw (0,1) -- (9,1) -- (9,0);
\draw (8,1) -- (8,0);
\draw[style=dashed] (-1,3) -- (2,6);
\draw[style=dashed] (-1,0) -- (5,6);
\draw[style=dashed] (1,-1) -- (8,6);
\draw[style=dashed] (4,-1) -- (10,5);
\draw[style=dashed] (7,-1) -- (10,2);
\draw[style=ultra thick] (0,5) -- (1,5) -- (1,4);
\draw[style=ultra thick] (1,3) -- (2,3) -- (2,2);
\draw[style=ultra thick] (6,2) -- (7,2) -- (7,1);
\draw[style=ultra thick] (8,1) -- (9,1) -- (9,0);
\draw (4,-1) node[below] {$P^+_\Lambda = (9,7,2,1,1)$};
\end{scope}

\end{tikzpicture}
\end{center}
\caption{An example illustrating the definition of displacement. Here $\Lambda = \{2\mod 3\}$.}
\label{figure:corners}
\end{figure}
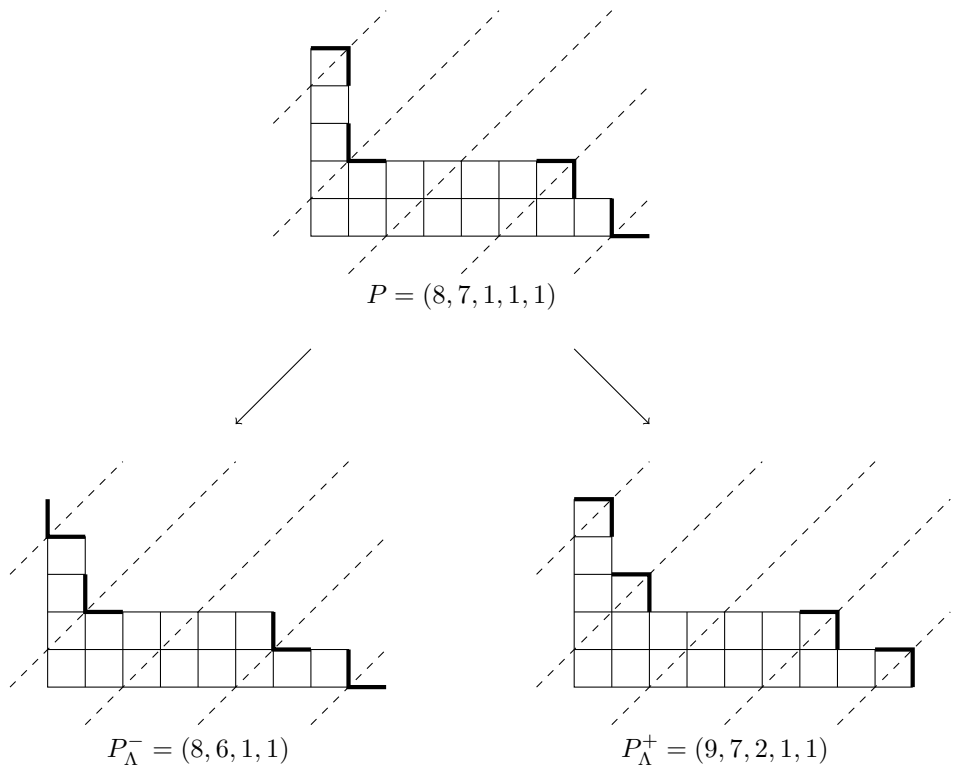

Observe that if $P'$ is any other partition such that $P^-_\Lambda \leq P' \leq P^+_\Lambda$, then the upward and downward displacements of $P'$ are the same as those of $P$ (with respect to $\Lambda$). So displacement can be regarded as a sort of projection to the nearest partition that is stable with respect to the given arithmetic progression.

Call two partitions $P_1, P_2$ \textit{linked} if there is an arithmetic progression $\Lambda$ (proper but possibly empty or singleton) such that $P_2$ is the upward displacement of $P_1$ and $P_1$ is the downward displacement of $P_2$. Note that this implies that $P_1$ is its own downward displacement and $P_2$ is its own upward displacement. Say that $P_1$ and $P_2$ are $k$-\textit{linked} if they are linked and $|P_2| - |P_1| = k$.

It is easy to verify that if $P_1, P_2$ are any two partitions with $P_1 \leq P_2$, then $P_1$ can be connected to $P_2$ by a sequence of $1$-linked partitions. Indeed, the arithmetic progressions can be taken to be singletons.

As we saw in the previous section, we are particularly interested in 2-linked partitions. More specifically, we are interested in partitions that can be joined by a path of 1-linked and 2-linked pairs, using as few 1-linked pairs as possible. Therefore make the following definition.

\begin{defn}
Call a sequence of partitions of increasing sum \emph{valid} if any two adjacent partitions in the sequence are $1$-linked or $2$-linked. Define the \emph{difficulty} $\delta(P)$ of a partition $P$ to be the fewest number of 1-linked adjacent pairs in a valid sequence from the empty partition to $P$.
\end{defn}

With this definition, we can now state the following lemma, which relates difficulties of partitions to dimensionally proper twisted Weierstrass points.

\begin{lemma}\label{displaceInduction}
Let $P$ be any partition and $\Lambda$ an arithmetic progression (proper and possibly empty or singleton). If $|P^+_\Lambda| - |P^-_\Lambda| \leq 2$ and $\tilde{\cw}_g(P^-_{\Lambda})$ has a dimensionally proper point lying in a fiber over $\cm_{g,1}$ of dimension $d$, then $\tilde{\cw}_g(P^+_\Lambda)$ has a dimensionally proper point lying in a fiber of dimension at most $(d+1)$, and at most $d$ if $P^+_\Lambda \neq P^-_\Lambda$.
\end{lemma}

\begin{proof}
Without loss of generality, let $P = P^-_\Lambda$. Let $(C,\cl,p) \in \tilde{\cw}_g(P)$ be a dimensionally proper point. By lemma \ref{twpgdr}, this can also be regarded as a dimensionally proper point of $\tilde{\cg}^r_{r+g, g}(a)$, where $r = g$, for $a_i = i + P_{r-i}$. Let $\Lambda' = \Lambda + (r+1)$. Then it follows that, again by lemma \ref{twpgdr}, $\tilde{\cw}_g(P^+_\Lambda) \cong \tilde{\cg}^r_{r+g,g} ( a^+_{\Lambda'})$. Since $a$ differs from $a^+_\Lambda$ in at most $2$ places, proposition \ref{displace} implies that $\tilde{\cw}_g(P^+_\Lambda)$ has a dimensionally proper point, lying in a fiber over $\cm_{g,1}$ of dimension at most $d+1$ (at most $d$ if $P^+_\Lambda \neq P^-_\Lambda$).
\end{proof}

\begin{cor}\label{diffToPoints}
Let $P$ be any partition. Then for all $g \geq \frac12 ( |P| + \delta(P))$, $\tilde{\cw}_g(P)$ has a dimensionally proper point, lying in a fiber over $\cm_{g,1}$ of dimension at most $\max(0, g - |P|)$.
\end{cor}

To prove theorem \ref{mainTheorem}, we are interesting in bounding the difficulty of ``box-shaped'' partitions, i.e. partitions of the form $(a^b)$. The table below shows some experimental data about the difficulties of these partitions for various values of $a$ and $b$.

\begin{center}$
\begin{array}{r|ccccccccccc}
&2&3&4&5&6&7&8&9&10&11&12\\\hline
2&2&4&4&6&6&6&6&8&8&10&10\\
3&4&5&6&7&6&7&8&7&6&7&6\\
4&4&6&4&6&6&8&4&6&6&6&4\\
5&6&7&6&7&6&5&6&5&4&5&6\\
6&6&6&6&6&6&4&4&4&4&4&4\\
7&6&7&8&5&4&7&4&5&6&5&6\\
8&6&8&4&6&4&4&4&4&4&6&4\\
9&8&7&6&5&4&5&4&5&4&5&4\\
10&8&6&6&4&4&6&4&4&4&4&6\\
11&10&7&6&5&4&5&6&5&4&7&4\\
\end{array}
$\end{center}

On the basis of these experimental data, we make the following conjecture.

\begin{conj}\label{boxConjecture}
There is a constant $C$ such that for all positive integers $a,b \geq 3$, $\delta((a^b)) \leq C$.
\end{conj}

The assumption $a,b \geq 3$ is harmless, since if $a=2$ then the corresponding twisted Weierstrass points detect $\fg^1_d$s, whose moduli are well-understood.

\begin{remark}
It is apparent from the definition that $\delta(P) = \delta(P^*)$ where $P^*$ is the conjugate partition. This is not surprising, in light of the duality $\tilde{\cw}_g(P) \cong \tilde{\cw}_g(P^*)$.
\end{remark}

\begin{remark}
Corollary \ref{diffToPoints} is equivalent, in the special case of box-shaped partitions, to saying that if $P = ((r+1)^{g-d+r})$, and $\rho \geq -g + \delta(P)$, then $\tilde{\cw}^r_{d,g}$ has a dimensionally proper point. So if conjecture \ref{boxConjecture} is true, it would should that that the ``phase transition'' from dimensionally proper to improper occurs very close to $\rho = -g$, as we expect. We suspect that conjecture \ref{boxConjecture} is tractable, but do not yet have a proof. In section \ref{linearSeries}, we prove a somewhat weaker bound, linear in $a$ and $b$, that is sufficient to give theorem \ref{mainTheorem}.
\end{remark}

\section{Primitive Weierstrass points}\label{primitive}

As an example and first application of the techniques described above, we will prove one of the main results of \cite{EH87} on the existence of dimensionally proper Weierstrass points.

To state the result requires a bit of terminology. A subset $S \subseteq \mathbb{Z}_{\geq 0}$ that contains $0$ and is closed under addition is called a \emph{numerical semigroup}. The size of the complement is called the \emph{genus}. The sum of the elements of the complement, minus $\binom{g+1}{2}$, is called the \emph{weight}. A semigroup is called \emph{primitive} if twice the smallest positive element is greater than all the gaps; this is equivalent to saying that for all sets $S' \geq S$ whose complement is size $g$, $S'$ is also a semigroup. Let $\cc_S \subseteq \cm_{g,1}$ be the locus of Weierstrass points with semigroup $S$; a point of $\cc_S$ is \emph{dimensionally proper} if the local codimension of $\cc_S$ in $\cm_{g,1}$ is equal to the weight of $S$.

\begin{thm}[Eisenbud and Harris]\label{EHTheorem}
Let $S$ be a primitive numerical semigroup of weight at most $g-2$. Then $\cc_S$ has dimensionally proper points\footnote{These theorem was improved by Komeda \cite{K}, who replaced $g-2$ by $g-1$; see remark \ref{KomedaRemark}.}.
\end{thm}

In fact the primitivity assumption is not an artifact of the proof, but is a crucial assumption. Our method also gives an easy proof of the following.

\begin{thm}\label{Imprim}
If $S$ is a non-primitive semigroup, then the moduli space $\cc_S$ of pointed curves with Weierstrass semigroup $S$ has no dimensionally proper points.
\end{thm}

Although this fact is not explicitly proved in \cite{EH87}, there are other ways to establish it; we give another argument in \cite{EW}, based on the notion of the effective weight of a semigroup.

First we re-express theorem \ref{EHTheorem} using the notation of this paper. Notice that $\cc_S \cong \tilde{\cw}_g(P)$, where $P$ is the partition given by $P_n = (g+n) - s_n$ (where $0 = s_0 < s_1 < s_2 < \cdots$ are the elements of $S$). That $S$ is a primitive semigroup is equivalent to saying that $2(g+1 - P_1) \geq (g+P^*_0)$ (where $P^*$ is the dual partition), which is equivalent, using the fact that $P_0 = g$, to $P_0 - P^*_0 \geq 2P_1 -2$. Thus theorem \ref{EHTheorem} follows from the following, by lemma \ref{diffToPoints}.

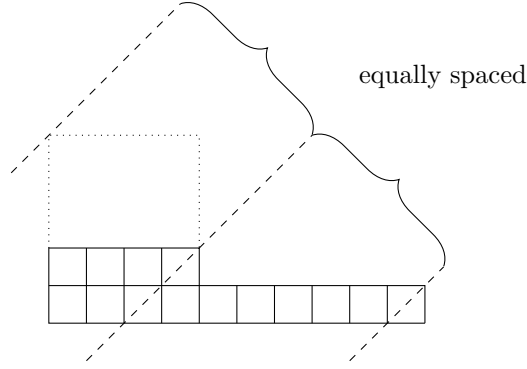
\begin{figure}
\begin{center}
\begin{tikzpicture}[scale=0.5]
\draw (0,2) -- (0,0) -- (10,0);
\draw (0,2) -- (4,2) -- (4,0);
\draw (1,2) -- (1,0);
\draw (2,2) -- (2,0);
\draw (3,2) -- (3,0);
\draw (0,1) -- (10,1) -- (10,0);
\draw (5,1) -- (5,0);
\draw (6,1) -- (6,0);
\draw (7,1) -- (7,0);
\draw (8,1) -- (8,0);
\draw (9,1) -- (9,0);
\draw[style=dotted] (0,2) -- (0,5) -- (4,5) -- (4,2);
\draw[style=dashed] (-1,4) -- (3.5,8.5);
\draw[style=dashed] (1,-1) -- (7,5);
\draw[style=dashed] (8,-1) -- (10.5,1.5);
\draw[decorate,decoration={brace, amplitude=0.4cm}] (3.5,8.5) -- (7,5);
\draw[decorate,decoration={brace, amplitude=0.4cm}] (7,5) -- (10.5,1.5);
\draw (8,6) node[above right] {$\mbox{equally spaced}$};
\end{tikzpicture}
\caption{Visual representation of the primitivity condition. Draw diagonals through the upper-right corners of the first two rows of the Young diagram, and also a third diagonal, equally spaced from the second. Then the rest of the Young diagram must fit below this third diagonal.}
\label{figure:prim}
\end{center}
\end{figure}

\begin{lemma}
Let $P$ be any partition such that $P_0 - P^*_0 \geq 2P_1 -2$ and $|P| \leq 2 P_0 -2$. Then $\delta(P) = 2P_0 - |P|$.
\end{lemma}
\begin{proof}
First, notice that any valid sequence of partitions ending in $P$ must have at least $P_0$ steps, since $P_0$ can increase by at most $1$ at each step. This means that $\frac12 (|P| + \delta(P)) \geq P_0$, i.e. $\delta(P) \geq 2P_0 - |P|$. So it suffices to show the opposite inequality.

The opposite inequality follows by induction on $P$. As the base case, consider the case $P_1 = 0$. Then $|P| = P_0 = \delta(P)$, so the result follows. So assume that $P_1 > 0$. Let $k \geq 1$ be the largest integer such that $P_k = P_1$. Then let $\Lambda$ be the arithmetic progression generated by $P_0-1$ and $P_k-k-1$. The corresponding diagonal lines meet the Young diagram of $P$ at only two corners, both outward, at the ends of rows $0$ and $k$ (see figure \ref{figure:prim}). Thus $P^+_\Lambda = P$ and $P^-_{\Lambda}$ differs in exactly two places from $P$: $P_0$ and $P_k$ are both decreaseed by $1$. Now, it is immediate that $|P^-_{\Lambda}| \leq 2 (P^-_{\Lambda})_0 -2$. It remains to show that $(P^-_{\Lambda})_0 - (P^-_{\Lambda})^*_0 \geq 2(P^-_{\Lambda})_1 -2$. Since $P_0$ decreased by $1$ under the displacement, the only way that this inequality could fail is if $P_0^*$ is unchanged, $P_1$ is unchanged, and the inequality was sharp before, i.e. $P_0 - P_0^* = 2P_1 -2$. This would mean that $P_1 = P_2$ and the Young diagram meets the third diagonal in figure \ref{figure:prim}; see figure \ref{figure:edge}. But in this case, we would have $|P| \geq P_0 + 2P_1 + (P^*_0-3) = 2P_0 -1$, which contradicts the assumption that $|P| \leq 2P_0-2$. Hence $P^-_{\Lambda}$ satisfies the hypotheses of the lemma. Also, it is clear that $2P_0 - |P|$ is unchanged and $\delta(P^-_\Lambda) \geq \delta (P)$, so the desired inequality follows by induction.
 completing the induction.\end{proof}

\begin{figure}
\begin{center}
\begin{tikzpicture}[scale=0.5]
\draw (0,5) -- (0,0) -- (10,0);
\draw (0,5) -- (1,5) -- (1,0);
\draw (0,4) -- (1,4);
\draw (0,3) -- (4,3) -- (4,0);
\draw (0,2) -- (4,2);
\draw (0,1) -- (10,1) -- (10,0);
\draw (2,3) -- (2,0);
\draw (3,3) -- (3,0);
\draw (5,1) -- (5,0);
\draw (6,1) -- (6,0);
\draw (7,1) -- (7,0);
\draw (8,1) -- (8,0);
\draw (9,1) -- (9,0);
\draw[style=dotted] (1,5) -- (4,5) -- (4,3);
\draw[style=dashed] (-1,4) -- (3.5,8.5);
\draw[style=dashed] (1,-1) -- (7,5);
\draw[style=dashed] (8,-1) -- (10.5,1.5);
\end{tikzpicture}
\caption{The only situation where primitivity may fail after displacement, in the proof of theorem \ref{EHTheorem}}
\label{figure:edge}
\end{center}
\end{figure}
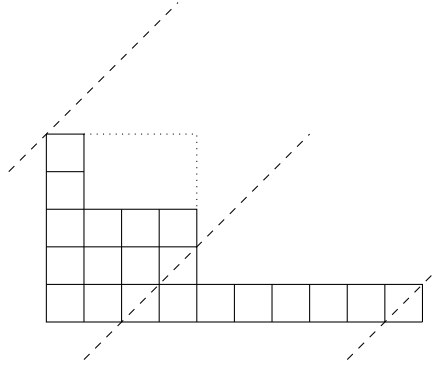

\begin{remark}\label{KomedaRemark}
Notice that the proof above very nearly shows the existence of all dimensionally proper Weierstrass points of weight less than $g$ (rather than $g-1$). If we attempt to prove this slightly stronger statement by an identical induction, we see that the inductive step fails only when $P_1 = P_2$, $P_3 = 1$, and $P_0^* = P_0-2P_1+2$ (i.e. the area enclosed by the dashed line in figure \ref{figure:edge} is empty). A different displacement works in this case, namely by turning in the first and last outward corner, unless $P_3 = 0$. So the only partitions that cannot be treated this way are $P = ((2m-1)\ m\ m)$. Komeda \cite{K} proved, by a different method, that dimensionally proper Weierstrass points corresponding to these partitions exist. So by adding these partitions as an additional base case, Komeda extended theorem \ref{EHTheorem} to all primitive semigroups of weight less than $g$.
\end{remark}

Using the same technique of displacement along elliptic curves, we can also prove the non-existence of dimensionally proper Weierstrass points. This result is substantially generalized, by a different method, in \cite{EW}. We include this proof because it demonstrated the capability of the technique of displacement to disprove the existence of dimensionally proper points as well.

\begin{proof}[Proof of theorem \ref{Imprim}]
Suppose for the sake of contradiction that $S = \{0, s_1, s_2, \cdots \}$ is a non-primitive semigroup such that $\cc_S$ has a dimensionally proper point. Let $P$ be the corresponding partition, so that $P_0 = g$ and $\tilde{\cw}_g(P)$ has a dimensionally proper point. Define $P^k$ to be the partition given by $P^k_0 = P_0 + k$ and $P^k_i = P_i$ otherwise. By displacing repeatedly along singleton arithmetic progressions, it follows that $\tilde{\cw}_{g+k} (P^k)$ has a dimensionally proper point (for each $k$). This corresponds to a dimensionally proper Weierstrass point in $\cc_{S^k}$, where $S^k = \{0, s_1+k, s_2+k, \cdots\}$. Since $S$ is not primitive, there exists a positive integer $f > 2s_1$ such that $f \not\in S$. Let $k = f - 2s_1$. Then $s_1 + k \in S^k$, but $2(s_1+k) = f+k \not\in S^k$. This is a contradiction; so $\cc_S$ cannot have any dimensionally proper points.
\end{proof}

\section{Special linear series}\label{linearSeries}

In order to prove the existence of a reasonably large class of dimensionally proper linear series, it suffices, by lemma \ref{diffToPoints} to bound the displacement difficulty of box-shaped partitions. We shall prove the following bound, which is likely to be very far from optimal, but is strong enough to give theorem \ref{mainTheorem}.

\begin{lemma} \label{boxlike}
Let $P$ be the partition $(a^b)$, i.e. the partition of the number $ab$ into $b$ equal parts, where $a,b \geq 2$. Then $\delta(P) \leq a + 3b -5 $.
\end{lemma}

The proof appears at the end of this section. This lemma, together with lemma \ref{diffToPoints} is sufficient to complete the proof of the main theorem.

\begin{proof}[Proof of theorem \ref{mainTheorem}]
Suppose that $g,d,r$ are integers such that $r \geq 1$, $g-d+r \geq 2$, and $0 > \rho \geq - \frac{r}{r+2} g + 3r - 3$. Let $a = g-d+r$, $b = r+1$, and $P = (a^b)$. Note that $a,b \geq 2$. By lemma \ref{twpToBN}, $\tilde{\cw}^r_{d,g}$ has a dimensionally proper point if and only if $\tilde{\cw}_g(P)$ has a dimensionally proper point. By lemmas \ref{diffToPoints} and \ref{boxlike}, it suffices to show the $g \geq \frac12 (ab + a + 3b -5)$. We may assume that $a,b \geq 3$ since the case $r=1$ is well-understood (by the duality mentioned in remark \ref{twpDuality}, the roles of $a$ and $b$ can be interchanged, so either can be taken to be $r+1$).

Now, $\rho \geq - \frac{r}{r+2} g + 3r -3$ is equivalent to $g - ab \geq \frac{b-1}{b+1} g + 3b - 6$, i.e. $\frac{2b}{b+1} g \geq ab + 3b - 6$. This is equivalent to $2g \geq  (a+3)(b+1) - \frac{6(b+1)}{b} = ab + a + 3b - 3 - \frac{6}{b}$. Since $b \geq 3$, this implies that $2g \geq ab + a + 3b - 5$, and the theorem follows.
\end{proof}

\begin{proof}[Proof of lemma \ref{boxlike}]
The proof will be by explicit construction of a sequence of partitions. First consider the case where \textbf{$a$ is even}. 

Define the following intermediate partitions: $P_{k,i} = (a^k\ (i+\frac12 a)\ i)$ (see figure \ref{figure:steps}), for $k \geq 0$ and $i \in \{0, 1, \cdots, \frac12 a\}$. 

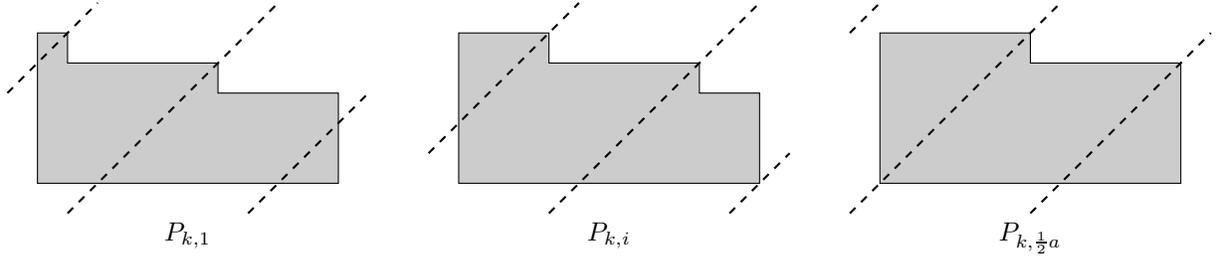
\begin{figure}
\begin{center}
\begin{tikzpicture}[scale=0.4]
\begin{scope}[xshift=-14cm]
\draw[fill=gray!40] (0,5) -- (1,5) -- (1,4) -- (6,4) -- (6,3) -- (10,3) -- (10,0) -- (0,0) -- cycle;
\draw[style=dashed,thick] (-1,3) -- (2,6);
\draw[style=dashed,thick] (1,-1) -- (8,6);
\draw[style=dashed,thick] (7,-1) -- (11,3);
\draw (5,-1) node[below] {$P_{k,1}$};
\end{scope}

\draw[fill=gray!40] (0,5) -- (3,5) -- (3,4) -- (8,4) -- (8,3) -- (10,3) -- (10,0) -- (0,0) -- cycle;
\draw[style=dashed,thick] (-1,1) -- (4,6);
\draw[style=dashed,thick] (3,-1) -- (10,6);
\draw[style=dashed,thick] (9,-1) -- (11,1);
\draw (5,-1) node[below] {$P_{k,i}$};

\begin{scope}[xshift=14cm]
\draw[fill=gray!40] (0,5) -- (5,5) -- (5,4) -- (10,4) -- (10,3) -- (10,3) -- (10,0) -- (0,0) -- cycle;
\draw[style=dashed,thick] (-1,5) -- (0,6);
\draw[style=dashed,thick] (-1,-1) -- (6,6);
\draw[style=dashed,thick] (5,-1) -- (11,5);
\draw (5,-1) node[below] {$P_{k,\frac12 a}$};
\end{scope}
\end{tikzpicture}
\end{center}
\caption{The intermediate partitions $P_{k,i}$ used in the proof of lemma \ref{boxlike}, together with the progressions $\Lambda_{k,i}$. The partition is it's own upward displacement for all values of $i$ except possibly one (shown in the middle).}
\label{figure:steps}
\end{figure}

Let $\Lambda_{k,i}$ denote the arithmetic progression generated by the two diagonals shows in figure \ref{figure:steps}. That is, $\Lambda_{k,i} = \{n:\ n \equiv i - k -2 \mod (\frac12a +1)\}$. Observe that if $1 \leq i \leq \frac12 a$, then $\Lambda_{k,i}$ does not meet the other outward-facing corner of the Young diagram, so it follows that

\begin{equation*}
\left( P_{k,i} \right)^-_{\Lambda_{k,i}} = P_{k,i-1} \mbox{ when } i>0.
\end{equation*}

Now consider the upward displacement. The only inward-turned corner that $\Lambda_{k,i}$ can meet is the one at the end of the first row of the Young diagram; this corresponds to the value $P_0 - 0 = a$.  From this we can conclude that

\begin{equation*}
\left( P_{k,i} \right)^+_{\Lambda_{k,i}} = P_{k,i} \mbox{ unless } k>0 \mbox{ and } a \equiv i - k -2 \mod (\frac12 a + 1).
\end{equation*}

For a fixed positive value of $k$, there is at most one value $i \in \{1, 2, \cdots, \frac12 a\}$ such that congruence above holds. Therefore the sequence of partitions

\begin{equation*}
P_{k,0} < P_{k,1} < \cdots < P_{k,\frac12a}
\end{equation*}

is nearly a valid sequence of partitions; at most one adjacent pair is invalid. By inserting an intermediate partition at that place (if necessary), we obtain a valid sequence of partitions with at most two steps increasing the sum by only $1$. Therefore $\delta(P_{k,\frac12a}) \leq 2 + \delta(P_{k, 0})$. For $k=0$, the original sequence is valid, so $\delta(P_{0,\frac12a}) \leq \delta(P_{0, 0})$.

Since $P_{k,\frac12a} = P_{k+1,0}$, it follows from this analysis that

\begin{equation*}
\delta(P_{b-1,0}) \leq 2(b-2) + \delta(P_{0,0}).
\end{equation*}

Now, $P_{b-1,0} \leq (a^b)$ with $|(a^b)| - |P_{b-1,0}| = \frac12 a$ and $|P_{0,0}| = \frac12 a$. From this it follows (by a sequence of displacements along singleton progressions) that

\begin{equation*}
\delta((a^b)) \leq a + 2b -4 \mbox{ when $a$ is even.}
\end{equation*}

Now, if $a$ is odd, then $\delta(((a-1)^b)) \leq a + 2b - 5$, and $((a-1)^b)$ can be linked to $(a^b)$ by a length $b$ sequence of length $b$. Therefore

\begin{equation*}
\delta((a^b)) \leq a + 3b - 5 \mbox{ when $a$ is odd.}
\end{equation*}

So whether $a$ is even or odd, $\delta((a^b)) \leq a + 3b - 5$.
\end{proof}

\end{document}